\documentclass[leqno]{amsart}
\usepackage{amsthm}
\usepackage{color}
\usepackage{amssymb}
\usepackage{amsmath,amsfonts,enumerate}
\usepackage{amsthm}

\numberwithin{equation}{section}

\newtheorem{thm}{Theorem}

\newtheorem{definition}[thm]{Definition}
\newtheorem{cor}[thm]{Corollary}

\begin{document}

\title{Multidimensional singular integrals and integral equations in fractional spaces  I}

 \author{$^{1}$N.K. Bliev}
 \email{bliyev.nazarbay@mail.ru}
\address{Institute of Mathematics and Mathematical Modeling, 050010, Almaty, Kazakhstan.}

\author{$^{2}$K.S. Tulenov}
\email{tulenov@math.kz}
\address{Al-Farabi Kazakh National University, 050040, Almaty, Kazakhstan;
Institute of Mathematics and Mathematical Modeling, 050010, Almaty, Kazakhstan.
}

\thanks{Corresponding authors: $^{2}$Al-Farabi Kazakh National University, 050040, Almaty, Kazakhstan;
$^{1,2}$Institute of Mathematics and Mathematical Modeling, 050010, Almaty, Kazakhstan. email:$^{1}$bliyev.nazarbay@mail.ru and $^{2}$tulenov@math.kz }



\date{}
\begin{abstract} In this paper, we investigate boundedness, invertibility and smoothness properties of multidimensional singular integral operators and solvability of the corresponding singular integral equations in Besov spaces.
\end{abstract}

\maketitle

\section{Introduction}
The work consists of two parts. This paper presents first part of the work. In this part, it is considered multidimensional singular operators with characteristics independent on poles and corresponding integral equations in Besov spaces ($B$-spaces). We obtain conditions of boundedness, differentiability and invertibility of considered singular operators, and solvability of some corresponding integral equations.

In the works of Calder\'{o}n and Zygmund \cite{CZ} and other authors \cite{H, M,S1,S2}, it were developed methods of real analysis, which allow to extend results on one dimensional singular integrals to multidimensional case by using methods of analytic functions of one variable. The main results of above cited papers on the theory of multidimensional singular integrals related to $L_p$-spaces. As far as we know, multidimensional singular integrals are considered for the first time in $B$-spaces of Besov in this paper, and it became possible after the research of previous authors. The presence in the scale of $B$-spaces of limit embedding theorems allows in some cases to achieve extremely limit results. For example, in the book \cite{B1}, the class of generalized analytic functions is extended by including continuous, in terms of $B$-spaces, functions, which in turn has a similar effect in the theories of singular integral equations with the Cauchy kernel \cite{B2}.
The main results of the paper are summarized in Theorems \ref{Th.1.1} and \ref{Th.1.2}, in which the conditions of boundedness, invertibility, and differentiability of the considered singular integral operators are obtained. In this paper, we used some results of Calder\'{o}n and Zygmund \cite{CZ}, L. H\"{o}rmander \cite{H} and materials from the theory of commutative normed rings \cite{GRSh} as well as information from the work \cite{M} of S.G. Mikhlin. The results of the paper will be used in the second part of the paper to study the (local and global) Noetherian conditions and more general singular integral operators, and the corresponding integral equations in $B$-spaces.

\section{Preliminaries}

 Let $E^n$ be a $n$-dimensional Euclidean space, $m$ is integer and $l>0$ is a number satisfying $m>l>0.$

By definition, a function $f$ belongs to the (isomorphic) Besov space $B^l_{p,\theta}(E^n), 1\leq p, \theta\leq\infty,$ if $f\in L_p(E^n)$ and the following seminorm is finite
$$||f||_{b^l_{p,\theta}(E^n)} = \Big{(} \int_{E^n} |h|^{-n-\theta l} ||\Delta^m_h f||^\theta_{L_p(E^n)} dh \Big{)}^{\frac{1}{\theta}},$$
where $\Delta^m_{h}$ is the finite difference of order $m$, $h\in E^n.$
In this space, the norm is defined by the following formula
$$||f||_{B^l_{p,\theta}(E^n)} = ||f||_{L_p(E^n)} + ||f||_{b^l_{p,\theta}(E^n)}.$$
For more details on this space we refer the reader to \cite{BIN,B}.
For commutative normed rings without a radical with symmetric involution (see \cite[pp. 57-59]{GRSh}), it is known from (see \cite[pp. 72-74)]{GRSh} that any non-invertible element is a generalized zero divisor.

\begin{definition}\label{D.1.1} An element $x$ in a commutative normed ring $R$ is called a generalized zero divisor, if there exists a sequence $\{y_n\}(\subset R)$ such that
\begin{enumerate}[{\rm (i)}]
\item $\inf_n ||y_n||>0,$

\item $\|x y_n\| \rightarrow 0\quad \text{as} \quad n \rightarrow \infty.$
\end{enumerate}
\end{definition}
If $h$ be a vector from $E^n,$ then we denote by $\tau_h$ the operator defined by
$$\tau_h f(x)=f(x+h).$$
\begin{definition}\label{D.2.1}
A linear bounded operator $A$ is called invariant with respect to the shift if
$$A(\tau_h f)=\tau_h (Af).$$
\end{definition}
Consider the operator $\alpha*f=f(\alpha x),$ where $\alpha\neq0$ is a real number.

\begin{definition}\label{D.3.1} An operator $ A $ will be called homogeneous, if for any $\alpha (>0)$ the following equality holds
\begin{equation}\nonumber
A (\alpha \ast f)=\alpha \ast (Af).
\end{equation}
\end{definition}

\section{Boundedness of singular integrals
with pole-independent characteristics}

$1^\circ.$ Consider the singular integral
$$Sf=\int_{E^n} \frac{\Omega(\theta)}{r^n} f(y) dy,\eqno(1.1)$$
where $r=|y-x|, \theta=\frac{y-x}{|y-x|},x,y \in E^n$
which is understood in the sense of the principal value of Cauchy
$$
Sf=\lim_{\varepsilon \rightarrow 0} \int_{r\geq \varepsilon} \frac{\Omega(\theta)}{r^n} f(y) dy.
$$
The function $ \Omega(\theta) $ is called the characteristic and $f (x)$ is called the density of the integral.
It is easy to verify directly that the singular integral operator $S$ given above is invariant (with respect to the shift) and homogeneous.
It follows from the result in \cite{CZ} of Calder\'{o}n and Zygmund that
\begin{enumerate}[{\rm (1)}]
\item $\Omega(\theta)$ is a homogeneous function of null degree, integrable over the unit sphere $S_1;$

\item $\int_{S_1} \Omega(\theta) d\theta=0;\quad \,\,\,\,\,\,\,\,\,\,\,\,\,\,\,\,\,\,\,\,\,\,\,\,\,\,\,\,\,\,\,\,\,\,\,\,\,\,\,\,\,\,\,\,\,\,\,\,\,\,\,\,\,\,\,\,\,\,\,\,\,\,\,\,\,\,\,\,\,\,\,\,\, \quad\quad\quad\quad\quad\quad\quad\quad\quad\quad\quad\quad(2.1)$

\item There is a number $\gamma(>1)$ such that $\int_{S_1} |\Omega(\theta)|^{\gamma} d\theta < \infty,$
\end{enumerate}
then the singular operator $S$ defined in (1.1) is bounded on $L_{p}(E^{n}), 1<p<\infty$, and the following inequality holds
$$||S||_{L_{p}(E^{n})} \leq C \Big{(} \int_{S_1} |\Omega(\theta)|^{\gamma} d\theta \Big{)}^{\frac{1}{\gamma}},\eqno(3.1)$$
where constant $C>0$ depends on $\gamma$. We everywhere further assume that conditions 1)-3) are fulfilled, i.e. the operator $S$ is bounded on $L_{p}(E^{n}), 1<p<\infty,$ and the inequality (3.1) holds. Throughout this paper, we denote by $\|\cdot\|_{B}$ and $\|\cdot\|_{p}$ the $B^l_{p,\theta} (E^n)$-norm and $L_p(E^n)$-norm, respectively.

\begin{thm}\label{Th.1.1} Operator $S$ defined by the singular integral (1.1) is bounded on $B^l_{p,\theta} (E^n)$, $1<p<\infty$, $1\leq\theta\leq\infty$, $l>0,$
besides its $B$-norm coincides with its $L_p$-norm
$$||Sf||_{B} \leq ||S||_p \cdot ||f||_{B}.\eqno(4.1)$$
\end{thm}
\begin{proof} The integral (1.1) is invariant with respect to the shift.
Indeed, for any $h\in E^n$ we have

$$
\tau_h (Sf)=\int_{E^n} \frac{\Omega \Big{(} \frac{y-x-h}{|y-x-h|} \Big{)} }{|y-x-h|^n} f(y) dy =\int_{E^n} \frac{\Omega \Big{(} \frac{y-x}{|y-x|} \Big{)} }{|y-x|^n} f(y+h) dy=S(\tau_h f).
$$

Therefore, the first difference is equal to
$$
\Delta_h (Sf)(x)=(Sf)(x+h)-(Sf)(x)=\int_{E^n} \frac{\Omega(\theta)}{r^n} \Delta_h f(y)dy.
$$

Consequently, for any natural number $m,$ we obtain

$$\Delta_h^m (Sf)(x)= \int_{E^n} \frac{\Omega(\theta)}{r^n} \Delta^m_h f(y) dy. \eqno(5.1)$$

Hence, it follows from (5.1) and (3.1) that

$$||Sf||_{B} = ||Sf||_p +(\int_{E^n} \frac{||\Delta^m_h (Sf)||_p^{\theta}}{|h|^{n+\theta l}} dh)^{\frac{1}{\theta}} \leq\\[0.2in]$$
$$\leq ||Sf||_p + ||S||_p (\int_{E^n} \frac{||\Delta^m_h (f)||_p^{\theta}}{|h|^{n+\theta l}} dh)^{\frac{1}{\theta}} \leq\\[0.2in]$$
$$\leq||S||_p \Big{(} ||f||_p + ||f||_{b^l_{p,\theta}(E^n)} \Big{)}=||S||_p \cdot ||f||_{B} \eqno(6.1)$$

We obtained that the singular operator
$S$ maps the space $B^l_{p,\theta} (E^n)$ into itself and $||S||_B\leq ||S||_p$. But, the norm of $Sf$ in $B^l_{p,\theta} (E^n)$ contains
$L_p(E^n)$-norm of $Sf$ as a term, therefore, we have $||S||_p \leq||S||_B$.
Consequently, $||S||_B=||S||_p,$ thereby completing the proof.
\end{proof}

\vskip0.2cm
$2^{\circ}$.
Let us conduct some necessary information from the paper \cite{H} of L. H\"{o}rmander in a slightly different, more convenient form for us.
We denote by $N_p^p$ the set of bounded from $L_p(E^n)$ into itself, invariant with respect to the shift, operators. In the book \cite[pp. 101-102]{H}, it was proved that all operators $A$ from $N_p^p$ belong to $N_{p'}^{p'}$ $(\frac{1}{p}+\frac{1}{p'}=1)$, also, they belong to $N_q^q$, where $q$ is a number between $p$ and $p'.$
Consequently, $A\in N^2_2$.
If we consider $A$ as an operator from $N^2_2$ and apply with the Fourier transform on it, then the operator $A$ becomes a multiplication operator on a measurable essentially bounded function, which we denote by $\widehat{A}$. Following S.G.~Mikhlin (see \cite{M}), we will call $\widehat{A}$ the symbol of the operator $A.$ Therefore, for each $A$ from $N_p^p$ $(1<p<\infty)$
corresponds a symbol $\widehat{A}\in L_{\infty}(E^n),$ where $L_{\infty}$ is the space of all essential bounded functions on $E^n.$

The space consisting of functions $f$ represented as $f= \widehat{A}, $ where $ A\in N^p_p $, is denoted by $M^p_p. $
For the norm in this space we take
$$||f||_{M^p_p}=||A||_{L_p(E^n)\rightarrow L_p(E^n)}.$$

As already noted, $M^p_p=M^{p'}_{p'} \subset M^2_2$. Besides (see \cite[pp. 100-102]{H})
$$||f||_{M^2_2} \leq ||f||_{M^p_p} = ||f||_{M^{p'}_{p'}}.\eqno(1.2)$$

We denote by $H^{p}_p $ the subspace $M^p_p,$ which is the closure in the space $M^p_p$ of a class of homogeneous of zero degree everywhere infinitely differentiable, except zero, functions. It is clear that the class $H_2$ coincides with the set of homogenous, continuous on the shpere $S_1,$ functions.
Note also that the functions in $ H_p $ are continuous everywhere, except for the point zero.

The space of operators $A$ such that $\widehat{A}\in H_p $ is denoted by $Q_p.$ The norm in $ Q_p $ is the usual operator norm. Obviously, $ H_p $ and $ Q_p $ are isomorphic commutative rings. It follows from the Theorem 2.7 in \cite{H} that there is a one-to-one correspondence between the set of points of the sphere $ S_1 $ and the set of maximal ideals $\mathfrak {M}_p $ of the ring $ Q_p $ such that for the corresponding to each other $l(\in \mathfrak{M}_p)$ and $\xi (\in S_1)$ holds the equality
$$l(A)= \widehat{A}(\xi),$$ where $A$ is an operator from $Q_p$.

This means that the value of the operator $A(\in Q_p) $ on the maximal ideal $ l(\in \mathfrak{M}_p) $ coincides with the value of the function symbol $\widehat{A} $ at the corresponding point $ \xi(\in S_1). $
Therefore, it is possible to identify maximal ideals for these rings with the corresponding points (see Remark in \cite[p. 32]{GRSh}). Thus, in the ring $ Q_p $, the only element vanishing on all maximal ideals is the identical zero. So, $ Q_p $ is a ring without a radical.

We introduce an involution into the ring $ Q_p $ as follows: Each operator $A(\in Q_p) $ is associated with the adjoint operator $ A^{\ast}. $ It follows from the definition of the space $Q_p$ that $A^{\ast} \in Q_p$ and
$\widehat {A}^{\ast}=\overline{\widehat{A}},$ where the trait means a sign of complex conjugation (see \cite[Theorem 1.34, p. 160]{M}). Hence, the involution introduced is symmetric.

It is known that the infinitely differentiable symbol of the singular operator (1.1) corresponds to an infinitely differentiable characteristic (see \cite[Theorem 2.32 , p. 155]{M}). Thus, it follows from the Theorem \ref{Th.1.1} ($ \|S\|_B=\|S\|_p $) that the singular integral operator $S$ in (1.1) belongs to $Q_{p}.$ The operators similar to $S$ were investigated in symmetric spaces of measurable functions and sequences in \cite{STZ, T}.
The following result holds (see (1.2)).

\begin{cor}\label{C.1.2}
The singular integral operator $S$ defined in (1.1), which is bounded in the space $B^l_{p,\theta}(E^n)$, $ 1 <p <\infty $, $ 1 \leq\theta\leq \infty $, $ l> 0 ,$ is bounded in $ B^l_{p',\theta}(E^n)$ $(\frac{1}{p'}+\frac{1}{p}= 1),$ is also in $ B^l_{q,\theta}(E^n)$ for all $q$ between $ p $ and $ p ',$ and therefore it is also bounded in $B^l_{2,\theta}(E^n).$
\end{cor}

\begin{thm}\label{Th.1.2} If the operator $A\in Q_p $ has no inverse, then $A$ is a generalized divisor of zero. An operator $ A\in Q_p $ is invertible if and only if
$ \widehat{A}(\xi) \neq 0 $ at all points $\xi$ belonging to the unit sphere $ S_1 .$
\end{thm}
\begin{proof} As indicated above, $ Q_p $ is a normalized commutative ring without a radical and with symmetric involution. Further, the proof of the theorem is similar to the proof of Theorem 1 in \cite[p. 73]{GRSh} and its Corollary 2 for normed rings.
\end{proof}

\vskip1.0cm

$3^{\circ}.$ Consider singular equation
$$Af=af+\frac{1}{(2\pi)^{\frac{n}{2}}}\int_{E^n} \frac{\Omega(\theta)}{r^{n}} f(y) dy=g(x),\eqno(1.3)$$
where $g(x)\in B^l_{p,\theta} (E^n)$, $1<p<\infty$, $1\leq\theta\leq\infty$, $l>0$, $a$ is independent of $x$, which is written as (see \cite[p. 109]{M})

$$Af=F^{-1}\widehat{A}Ff=g(x),\eqno(1.3')$$
where, $\widehat{A}(\xi)$ is the symbol of $A$ (see (2.1))

$$\widehat{A}(\xi)=a+F(\frac{f(\theta)}{r^{n}}),$$
Here, $F$ is the Fourier transform in $E^n.$

Assuming that $ \widehat{A}(\xi) \neq 0 $ for all points $\xi\in S_{1}, $ it follows from the Theorem \ref{Th.1.2} that there is a bounded inverse operator $C=A^{-1}\in Q_ {p} $ and a bounded (continuous) inverse function $\widehat{A}^{-1}(\xi)$ which is the symbol of the operator $C.$ Consequently, according to the rule of multiplication of symbols, the equation (1.3 ') ((1.3)) has a unique in the space $ B^l_{p,\theta}(E^n),$ $ 1<p<\infty,$ $1\leq\theta\leq\infty,$ $l>0$ solution.
Moreover, this solution is expressed by the formula
$$f=F^{-1}A^{-1}(\xi)Fg.$$
A more general equation
$$F^{-1}\widehat{A}Ff+Tf=g(x),\eqno(2.3)$$
where $T$ is a completely continuous operator in $B^l_{p,\theta}(E^n), $ is equivalently reduced to an equivalent Fredholm type equation
$$f+CTf=Cf.$$

Now consider the following integral equation
$$f(z)-\mu(z)(\Pi f)(z)=g(z), \eqno (3.3)$$
where, $(\Pi f)(z)=-\frac{1}{\pi}\int\int_{E_{2}}\frac{f(\xi)}{(\xi-z)^{2}}d\zeta d\eta, z=x+iy, \xi=\zeta+i\eta$, $E^{2}$ is a two dimensional plane, $g(z)\in B^l_{p,\theta}(E^2),$ $1<p<\infty,$ $1\leq\theta\leq\infty$, $l>0,$ and $\mu(z)$ is a measurable function satisfying inequality

$$\mid\mu(z)\mid\leq q< 1. \eqno(4.3)$$

In order to avoid possible reservations, we assume that $\mu(z)$ is a finite infinitely differentiable function.

By Theorem \ref{Th.1.1} ($n=2$), the singular integral $(\Pi f)(z)$ is bounded in $ B^l_{p,\theta}(E^2),$ and it is used in the theory of generalized analytic functions \cite{V}, the problem of finding homeomorphisms of the (elliptic) system of Beltrami equations (\cite{B, V}), which is written as a complex equation

$$\frac{\partial W}{\partial\bar{z}}-\mu(z)\frac{\partial W}{\partial z}=0. \eqno(5.3)$$

If $(\Pi f)(z)$ is presented in the form
$$\Pi f=-\frac{1}{\pi}\int\int_{E^{2}}\frac{e^{-2i\theta}}{r^{2}}f(\xi)d\zeta d\eta,$$
where, $\theta=arg(\xi-z), r=\mid z-\xi\mid.$ Hence, it can be seen that $\Pi\in Q_{p}.$
According to Corollary \ref{C.1.2}, the operator $\Pi $ maps $ B^l_{p,\theta}(E^2)$ into itself and is bounded (see (1.2)). Moreover, its $ B $-norm coincides with its $L_{2}(E^{2})$-norm that is equal to one ([10, Chapter I.9]). There is a bounded inverse operator $ (J-\mu\Pi)^{-1}$, which acts on the scale of the spaces $ L_{p}(E^{2}),$ $ 1<p<\infty$, \cite{K,Vi}. Therefore, it follows from the Theorem \ref{Th.1.1} that the integral equation (3.3) is solvable in the space $ B^l_{p,\theta}(E^2) $ and this solution $ f (z)(\in B^l_{p,\theta}(E^2))$ is unique. Then, the expression
$$W(z)=z-\frac{1}{\pi}\int\int_{E^{2}}\frac{f(\xi)}{\xi-z}d\zeta d\eta $$
is the principal (complete) homeomorphism of the Beltrami equation (5.3) ([10, Section 3,5]).

\section{Acknowledgment}

This work was partially supported by the grant No. AP05133283
of the Science Committee of the Ministry of Education and Science of the Republic of Kazakhstan.


\begin{thebibliography}{99}

\bibitem{A} P.S. Alexandrov, {\it Combinatorial Topology}, OGIZ, Moscow, 1947. [In Russian]

\bibitem{At} F. V. Atkinson, {\it The normal solvability of linear equations in normed spaces}, Mat. Sb. (N.S.), {\bf 28(70)}:1 (1951), 3--14.


\bibitem{BIN} O.V. Besov, V.P. Il'in, and S.M. Nikol'ski\u{i}, {\it Integral representations of function and embedding theorems, 1, 2.}, John Willey, New York, 1978; 1979.

\bibitem{B1} Bliev N., {\it Generalized analytic functions in fractional spaces}, USA, Boston, Addison Wesley Longman Inc., 1997.

\bibitem{B2} Bliev N.K.,  {\it Singular integral operators with a Cauchy kernel in fractional spaces},  Siberian Mathematical Journal, {\bf 47}:1 (2006), 28--34.

\bibitem{B} B. V. Boyarsky, {\it Generalized solutions of a system of differential equations of first order and of elliptic type with discontinuous coefficients}, Mat. Sb. (N.S.), {\bf 43(85)}:4 (1957), 451--503.

\bibitem{CZ} A.P. Calder\'{o}n and A. Zygmund, {\it On singular integrals}, Amer. J. Math., {\bf 78}:2 (1956), 289--309.

\bibitem{Hal} P.R. Halmos, {\it Measure theory}, D. Van Nostrand, New York, 1950.

\bibitem{H} L. H\"{o}rmander, {\it Estimates for translation invariant operators in $L^p$ spaces}, Acta Math., {\bf 104} (1960), 93--140.

\bibitem{GRSh} I.M. Gelfand, D. A. Raikov, and G.E. Shilov, {\it Commutative Normed Rings}, "Fizmatgiz", Moscow, 1960. [In Russian]

\bibitem{K} I. I. Komyak, {\it On the solvability of a class of two-dimensional singular integral equations}, Dokl. Akad. Nauk SSSR, {\bf 250}:6 (1980), 1307--1310.

\bibitem{M} S.G. Mikhlin, {\it Multidimentional Singular Integrals and Integral Equations}, "Fizmatgiz", Moscow, 1962. [In Russian]

\bibitem{S1} I.B. Simonenko, {\it A new general method of investigating linear operator equations of singular integral equation type. I}, Izv. Akad. Nauk SSSR Ser. Mat., {\bf 29}:3 (1965), 567--586.

\bibitem{S2} I.B. Simonenko, {\it A new general method of investigating linear operator equations of singular integral equation type. II}, Izv. Akad. Nauk SSSR Ser. Mat., {\bf 29}:3 (1965), 757--782.

\bibitem{STZ} F. Sukochev, K. Tulenov, and D. Zanin, {\it The optimal range of the Calder\'{o}n operator and its applications.} J. Func. Anal. (2019), 1--47. doi.org/10.1016/j.jfa.2019.05.012 (In Press)

\bibitem{V} I.N. Vekua, {\it Generalized Analytic Function}, Oxford, New York, Pergamon Press, 1962.

\bibitem{Vi} V. S. Vinogradov, {\it On the solvability of a singular integral equation}, Dokl. Akad. Nauk SSSR, {\bf 241}:2 (1978), 272--274.

\bibitem{T}  K.S. Tulenov, {\it The optimal symmetric quasi-Banach range of the discrete Hilbert transform.} Arch. der Mathematik (2019), 1-12, DOI: 10.1007/s00013-019-01375-w. (In Press)

\end{thebibliography}
\end{document}